\documentclass[12 pt,letterpaper,reqno]{amsart}
\usepackage{amssymb,latexsym,amsthm,amsmath,amsfonts, mathtools, multirow, longtable,comment}
\usepackage[hidelinks]{hyperref}
\DeclareFontFamily{U}{wncy}{}
\DeclareFontShape{U}{wncy}{m}{n}{<->wncyr10}{}
\DeclareSymbolFont{mcy}{U}{wncy}{m}{n}
\DeclareMathSymbol{\Sh}{\mathord}{mcy}{"58}

\newcommand{\genlegendre}[4]{%
	\genfrac{(}{)}{}{#1}{#3}{#4}%
	\if\relax\detokenize{#2}\relax\else_{\!#2}\fi
}
\newcommand{\legendre}[3][]{\genlegendre{}{#1}{#2}{#3}}

\usepackage{hyperref}
\hypersetup{
	colorlinks,
	citecolor=blue,
	filecolor=black,
	linkcolor=red,
	urlcolor=black
}

\usepackage{mathdots}

\long\def\symbolfootnote[#1]#2{\begingroup
	\def\thefootnote{\fnsymbol{footnote}}\footnote[#1]{#2}\endgroup}
\newcommand{\Z}{\mathbb Z}
\newcommand{\Q}{\mathbb Q}

\makeatletter
\def\imod#1{\allowbreak\mkern10mu({\operator@font mod}\,\,#1)}
\makeatother
\newtheorem{theorem}{Theorem}[section]
\newtheorem{lemma}[theorem]{Lemma}
\newtheorem{corollary}[theorem]{Corollary}
\newtheorem{proposition}[theorem]{Proposition}
\newtheorem*{theorem*}{Theorem}
\theoremstyle{definition}
\newtheorem{remark}[theorem]{Remark}

\numberwithin{equation}{section}
\newtheorem{notation}[theorem]{Notation}

\title[A necessary condition for a congruent number of the form $8k+3$]{A necessary condition for a congruent number of the form $8k+3$}

\author{Shamik Das, Sudipa Mondal}
\address{Shamik Das, shamik93das@gmail.com, Indian Institute of Technology Kanpur, India}
\address{Sudipa Mondal, sudipa.mondal123@gmail.com, Indian Institute of Technology, Madras (Chennai), India}

\begin{document}

\begin{abstract}
A positive square-free integer is called a \textit{congruent number} if it arises as the area of a right triangle with rational side lengths. Let $ n = p_1p_2 \cdots p_t q $ be a square-free integer, where each $ p_i \equiv 1 \pmod{8} $ and $ q \equiv 3 \pmod{8} $, with the $ p_i $ and $ q $ being distinct primes. In this article, we present a congruence relation modulo powers of 2 between the 2-part of the class numbers of $ \mathbb{Q}(\sqrt{-n}) $ and $ \mathbb{Q}(\sqrt{-p_1p_2 \cdots p_t}) $, under the assumption that $ n $ is a congruent number, using a modified R\'edei matrix.
 \end{abstract}
	

    \maketitle	
	
\section{Introduction}\label{intro}
\noindent A positive integer $n$ is a congruent number if there exists $a,b,c \in \Q^{\times}$ such that 
\begin{equation*}
    a^2+b^2 = c^2, \text{ and } n=\frac{1}{2}ab.
\end{equation*}
The \textit{congruent number problem} seeks a characterization of all congruent numbers and remains a challenging problem, as no general algorithm is currently known for its resolution. In 1640, Fermat demonstrated that 1 is not a congruent number, reducing the problem to identifying all square-free numbers that are congruent numbers. It is clear that  $n$  is a congruent number if and only if its square-free part is a congruent number. For a  square-free positive integer $n$, consider the elliptic curve
\begin{equation}
\label{cec}
E_{n}: y^2=x^3-n^2x.
\end{equation}
The solutions of equation (\ref{cec}) with $y=0$ are $\{(0,0),(n,0),(-n,0)\}$. Let $y \neq 0$. Then there is a one-to-one correspondence between the following two sets
: \begin{equation*}
    \{(a,b,c): a^2+b^2=c^2, n=\frac{1}{2}ab\}, \quad \{(x,y): y^2=x^3-n^2x,\; y\neq 0\}.
\end{equation*}
In fact, it can be concluded that $n$ is a congruent number if and only if the elliptic curve (\ref{cec}) has a point of \textit{infinite order}. The curve $E_n$ as in (\ref{cec}) is referred to as a \textit{congruent number elliptic curve}. The $\Q$-rational points $E_{n}(\mathbb{Q})$ is finitely generated abelian group due to Mordell-Weil theorem, that is, $E_n(\Q) \cong E_n(\Q)_{\mathrm{tors}} \oplus \mathbb{Z}^{r_n}$. The number $r_n$ is called the algebraic rank of $E_{n}(\mathbb{Q})$. The (weak) Birch and Swinnerton-Dyer conjecture predicts that $r_n$ and the order of vanishing  of the Hasse-Weil $L$-function $L(E_n,s)$ at $s=1$, denoted by $R(n)$ are the same. Based on the computations conducted by Birch and Stephens \cite{bs} we have that 
\begin{equation*}
R(n) \equiv
\begin{cases}
      0 \pmod 2 ~~\text{ if } & n \equiv 1,2,3 \pmod 8 \\
     1 \pmod 2 ~~ \text{ if } & n \equiv 5,6,7 \pmod 8.  
\end{cases}  
\end{equation*}
Consequently, the case where $ n \equiv 1, 2, 3 \pmod{8} $ becomes more intriguing, since under the assumption of the BSD conjecture, we lack predictions analogous to those for $ n \equiv 5, 6, 7 \pmod{8} $. Let $q_i$ be a prime number such that $q_i \equiv i \pmod 8$. Stephens \cite{stephens} proved that the primes $q_5, q_7$ are congruent numbers. Heegner \cite{heegner} and Birch \cite{birch} showed that $2q_3$ and $2q_7$ are congruent numbers. Genocchi \cite{genocchi} proved that $2q_5, q_3$ are non-congruent numbers. Lagrange \cite{lagrange} showed that $q_1q_3$ is a non-congruent number if $q_1$ is quadratic non-residue modulo $q_3$.  
Tunnell \cite{tunnell} used the theory of half-integral weight modular forms to provide a criterion for $q_1$ and $2q_1$ to be non-congruent numbers. \smallskip

For a square-free positive number $n$, we denote the ideal class group and the ideal class number of the imaginary quadratic number field $\Q(\sqrt{-n})$ by $\mathcal{C}(-n)$ and $h(-n)$, respectively. For $k\geq 1$, the $2^k$-rank of $\mathcal{C}(-n)$ is defined in (\ref{rankofclassgrp}). In \cite{qinli}, Qin and Li provide some conditions on the $8$-rank of the class group of the imaginary quadratic field 
$\Q(\sqrt{-n})$ with $n>0$ to produce families of non-congruent numbers.
Suppose $p$  and  $q$  are distinct primes such that $p \equiv 1 \pmod{16}$ and $q \equiv 3 \pmod{8}$, with  $q$  being a quadratic residue modulo $p$. Das \cite{Das} demonstrated that $pq$  is not a congruent number if $h(-2pq) \not\equiv p-1 \pmod{16}$. Qin \cite{qin} developed a novel Shimura lift approach and, based on this, established a new criterion for congruent numbers. As one of its applications, in the case modulo 8 with residue 3, \cite{qin} proved that if  $h(-pq) \not\equiv h(-p) \pmod{8}$, then $pq$ is not a congruent number, where $p$ and $q$ are primes with $p \equiv 1 \pmod{8}$ and $q \equiv 3 \pmod{8}$, as shown using the Minkowski--Siegel formula. In this article, the authors generalise this result of Qin for a square-free integer $n$ where $n=p_1p_2\cdots p_tq,\,t>0$ with $(p_i,q) \equiv (1, 3) \pmod 8$. Although the approach in this article is substantially different from the method used in \cite{qin}. Let $n=p_1p_2\cdots p_tq ~~(t > 0)$ be a square-free positive integer such that $p_i \equiv 1 \pmod8$ and $q \equiv 3 \pmod{8}$ are primes.
    \begin{enumerate}
        \item We denote $n_q:= \frac{n}{q}$.
        \item Define a matrix $A_n=[a_{i,j}]_{t \times t}$ by 
   \begin{equation}
   \label{matrixA}
   a_{i,j}=\begin{cases}
   	0	\text{ if } \legendre{p_j}{p_i}=1 \textit{ with } j\neq i\\
   	1	\text{ if } \legendre{p_j}{p_i}=-1 \textit{ with } j\neq i \end{cases} \text{ and } a_{i,i}=\sum_{j; j\neq i} a_{i,j},
    \end{equation}
    \end{enumerate}
where $\legendre{\cdot}{\cdot}$ is the Legendre symbol. For $t=1$, we can consider $A_n= [0]_{1 \times 1}$. The main theorem of this article is the following:
\begin{theorem}
\label{thmPq}
    Let $n=n_{q}q$ be defined as above. Assume the following: 
	\begin{enumerate}
		\item[(i)] $\legendre{q}{p_i}=1$ with $i \in \{1, 2, \ldots, t\}$,
		\item[(ii)] $\operatorname{rank} A_n=t-1$.
	\end{enumerate}	
	If $n$ is a congruent number, then we have 
	$	h(-n)\equiv h(-n_q) \pmod{2^{t+2}}$.
\end{theorem}
 It is noteworthy that if $n$  satisfies conditions (i) and (ii) of Theorem \ref{thmPq}, then $h(-n) \equiv h(-n_q) \equiv 0 \pmod{2^{t+1}}$ (see \cite{lu} and \cite{wy}). Theorem \ref{thmPq} implies that  $h(-n)$  and $h(-n_q)$ remain in the same congruence classes modulo $2^{t+2}$ if $n$  is a congruent number. To prove Theorem \ref{thmPq}, we first examine the image of Mordell-Weil group under the $2$-descent map in \S \ref{preli}. We also exploit the $8$-rank of the class groups of $\Q(\sqrt{-n})$ and $\Q(\sqrt{-n_q})$ and prove some auxiliary results in \S \ref{4rank8rank}. In \S \ref{proofofthemaintheorem}, we prove the main theorem.  Finally in the last section, we give some examples as an evidence for  Theorem \ref{thmPq}. 

    \section{ Monsky matrix and 2-Selmer rank }\label{preli}

\noindent In this section, we introduce the Mordell-Weil group and $2$-Selmer rank of a congruent number elliptic curve $E_{m}$. We then explore the relationship between the 2-Selmer rank and the rank of the Monsky matrix associated with a square-free positive integer. \smallskip 

Consider the congruent number elliptic curve $E_m: y^2=x^3-m^2x$ where $m \in \mathbb{N}$ with $m$ being square-free. 
It is known that $E_m(\Q)\cong \Z^{r_m} \oplus \Z/2\Z \oplus \Z/2\Z$, where $r_m$ is the rank of the elliptic curve. Hence 
$E_m(\Q)/2E_m(\Q) \cong (\Z/2\Z)^{r_m+2}.$ Consider the $2$-Selmer group ${\mathrm{Sel}}_2(E_m/\Q)$ and the Shafarevich-Tate group $\Sh(E_m/\Q)$ as defined in \cite[Chapter X]{silv}. Then we have the following exact sequence \begin{equation}\label{exactenq}
	0 \to E_m(\Q)/2E_m(\Q) \to {\mathrm{Sel}}_2(E_m/\Q) \to \Sh(E_m/\Q)[2] \to 0. 
 	\end{equation} 
If $\#{\mathrm{Sel}}_2(E_m/\Q) = 2^{s_m+2}$, then we have $0 \leq r_m \leq s_m$. The integer $s_m$ is said to be the $2$-Selmer rank of $E_m$. Selmer conjectured that $r_m$ and $s_m$ have the same parity. Hence the $2$-Selmer rank $s_m$ plays an important role to estimate that Mordell-Weil rank $r_m$ of the Elliptic curve $E_m$ over $\mathbb{Q}$. In \cite[Appendix]{hb}, Monsky showed that $s_m$ is even if $m\equiv 1,2,3 \pmod 8$ and, is odd if $m \equiv 5,6,7 \pmod 8$. Hence we cannot predict for $m$ being a congruent or non-congruent number when $s_m$ is even assuming the Selmer conjecture. \smallskip 
    
Let $m=p_1p_2\cdots p_r$ be a square-free natural number where $p_i$'s are odd primes. For $l \in \{2,-2\}$, we define a $r \times r$ diagonal matrix $D_l=[d_{i,l}]$  and another  $r \times r$ matrix $C=[c_{i,j}]$ respectively as follows: \begin{equation*}
    d_{i,l}=\begin{cases}
    0	\text{ if } \legendre{l}{p_i}=1, \\
    1	\text{ if } \legendre{l}{p_i}=-1, \end{cases} c_{i,j}=\begin{cases}
    0	\text{ if } \legendre{p_j}{p_i}=1 \textit{ with } j\neq i,\\
    1	\text{ if } \legendre{p_j}{p_i}=-1 \textit{ with } j\neq i, \end{cases} \text{ and } c_{i,i}=\sum_{j; j\neq i} c_{i,j}. \end{equation*}
    The Monsky matrix is defined as follows:
    \begin{equation}
    \label{Monsky_matrix}
      M=\begin{bmatrix}
    	C+D_2 & D_2 \\
    	D_2 & C+D_{-2}
    \end{bmatrix}.   
    \end{equation}
    \begin{lemma}\cite[Appendix]{hb}\label{sm}
       Let $s_{m}$ be the 2-Selmer rank of the elliptic curve $E_{m}$ over $\Q$ and $M$ be the corresponding Monsky matrix as defined in \eqref{Monsky_matrix}. Then $s_m= 2r- {\operatorname{rank}(M)}$.
    \end{lemma}
\noindent  Consider the set $G=\{a:a>0,\; a|m \}$. It is easy to check that $G$ forms a group under the following binary operation:
    \begin{equation}\label{defineG}
    a * b=\frac{ab}{(a,b)^2} , ~~~ \, \, a,b \in G.
    \end{equation}
Note that $G \cong (\Z/2\Z)^{r}$. For each prime $p\,|\,m$, define a homomorphism $\phi_p$ by \begin{equation}
    	\begin{split}
    	\phi_p: G \times G & \to \{\pm 1\} \times \{\pm 1\} \\
    	\phi_p(a,b) & = \left(\legendre{a}{p}, \legendre{b}{p} \right) \text{ for } p \nmid ab, \\
    	\phi_p(1,m) = &\left(\legendre{2}{p}, \legendre{-2}{p} \right),~~~
    	\phi_p(m,1)  = \left(\legendre{2}{p}, \legendre{2}{p} \right). 
        \end{split}
    \end{equation}
   Also consider the following system of equations 
   \begin{equation}\label{congeq}
   	\begin{split}
   	abx^2+my^2 & = az^2 \\
    abx^2-my^2 & = bw^2
   \end{split}
   \end{equation}
   The system of equations (\ref{congeq}) are solvable in $\Q_p$ for each $p$ if $(a,b) \in {\mathrm{ker}(\phi_p)}$ for each $p$. 
   \begin{lemma}
   \label{expressK}
   	Let $K:= \bigcap\limits_{p\,|\,m} {\mathrm{ker}(\phi_p)}$. Then $\# K= 2^{s_m}$.
   \end{lemma}
 \begin{lemma}
\label{dequa_rank_Selmer}
There are exactly $2^{r_{m}}$ systems in (\ref{congeq}) with non-trivial integer solutions. Moreover, there are $2^{s_{m}}$ systems in (\ref{congeq}) which are everywhere locally solvable.
\end{lemma}  
The proofs of the above lemmas can be found in \cite{hb2} and \cite{hb}. Let $n$ be as in Theorem \ref{thmPq}. We recall the matrix $A_{n}$ defined as in (\ref{matrixA}). The following crucial result plays a significant role in proving our main theorem.
 \begin{proposition}
    \label{expK}
    Suppose that $n$ is defined as in Theorem \ref{thmPq}. Then we have that $s_n=2$ and $K= \{(1,1),(n_q,1),(1,n_q),(n_q,n_q)\}$, where $n_q=\frac{n}{q}$ and $K$ is as in Lemma \ref{expressK}.
    \end{proposition}
    
    \begin{proof}
           It is easy to see that, in this case, the Monsky matrix $M$ (see (\ref{Monsky_matrix})) has the form $$M= \begin{bmatrix}
    	R & S \\
    	S & T
    	\end{bmatrix},$$ where 
     $R,S,T$ are $(t+1) \times (t+1)$ matrices with $R=S+T$,  $S=[{s}_{i,j}], \; T=[{t}_{i,j}]$, such that 
   \begin{equation*}
    {s}_{i,j}=\begin{cases}
    1	\quad \text{ if } i=j=t+1, \\
    0	\quad \text{ otherwise, } \end{cases} \quad \text{ and, }{t}_{i,j}=\begin{cases}
    a_{i,j}	~~\text{ if }  1 \leq i,j \leq t, \\
    0	\quad \text{ otherwise.} \end{cases}
   \end{equation*}

    The numbers $a_{i,j}$ are defined as in (\ref{matrixA}). As ${\operatorname{rank}}(A_{n})=t-1$, we have that ${\operatorname{rank}}(M)=2t$. Since there are $(t+1)$ number of primes in the prime factorazition of $n$, by Lemma \ref{sm} we get $s_n=2$. Note that, for each prime $p$, the set $\{(1,1),(n_q,1),(1,n_q),(n_q,n_q)\} \subseteq {\mathrm{ker}(\phi_p)} \subseteq K$. Using Lemma \ref{expressK}, we have the desired result.
    \end{proof}

\begin{remark}
\label{rem_Sel_to_Mordell}{\rm 
Suppose that $n$ is defined as in Theorem \ref{thmPq}. Assume that we have the Mordell-Weil rank $r_{n} \geq 1$. So, we have $s_{n} \geq 1$. Then by Lemma \ref{dequa_rank_Selmer}, the system of equations (\ref{congeq}) has non-trivial and pairwise coprime integral solutions $(x,y,z,w)$ for $n$ for at least one non-trivial pair $(a,b)$. Noting that the existence of a global solution to the system of equations (\ref{congeq}) implies the existence of its local solution, we can conclude that the system of equations (\ref{congeq}) has a non-trivial solution for at least one non-trivial pair $(a,b) \in K$, where $K$ is defined in Lemma \ref{expressK}}.
\end{remark}

\section{4-rank and 8-rank of the class group} \label{4rank8rank}
\noindent In this section, we provide the required formula to compute the $8$-rank for the imaginary quadratic field $F=\mathbb{Q}(\sqrt{-n})$, where $n$ is defined as in Theorem \ref{thmPq}. Our focus turns to a square-free positive integer $P = p_{1}p_{2} \cdots p_{t}$, where all prime numbers $p_{i}$ satisfy $p_{i} \equiv 1 \pmod{8}$. Notably, $\pm 2$ is a quadratic residue modulo $p_{i}$ for each $i$, implying that $p_{i}$ splits in the quadratic field $\mathbb{Q}(\sqrt{\pm 2})$, which has class number $1$. This ensures that $P$ can be expressed as
\begin{equation}
\label{expP}
P = 2e^{2} - f^{2} = u^{2} + 2v^{2},
\end{equation}
where $e$, $f$, and $u$ are odd, and $v$ is even, and all are positive. The following lemma establishes a relation between the numbers $e$ and $v$. 

\begin{lemma}
\label{rebeeandv}
Let $P, e , v$ be defined as in (\ref{expP}). Then $\legendre{-1}{e}=1 \iff v \equiv 0 \pmod{4}$. 
\end{lemma}
    \begin{proof}
    It is easy to see that $e,f,u$ are odd, $v$ is even and $u^2\equiv f^2\equiv 1 \pmod{8}$. From (\ref{expP}), we have $2(e+v)(e-v)=u^2+f^2$. 
	
    Let $\gcd(u,f)=d$ and $l$ be an odd prime factor of $u^2+f^2$. If $l \nmid d$, then $l \equiv 1 \pmod{4}$. Indeed $u^2+f^2 \equiv 0 \pmod l$ which implies that $\legendre{-1}{l}=1$ and this is true if and only if $l \equiv 1 \pmod{4}$. 
  
 If $l\mid d=\gcd(u, f)$,  then $l^2 \, | \, (e+v)(e-v)$. Note that $l$ divides exactly one of $e+v$ and $e-v$, since $\gcd(e,f)=1$. Hence $l^2$ divides exactly one of $e+v$ and $e-v$ and $l^2 \equiv 1 \pmod 4$. Thus $e\pm v\equiv 1 \pmod{4}$. Hence $e \equiv 1 \pmod{4}  \iff v \equiv 0 \pmod{4}$.
\end{proof}
	
\noindent Let $m$ be a square-free positive integer and the ideal class group $\mathcal{C}(-m)$ of $\Q(\sqrt{-m})$. For $k \geq 1$, the $2^k$-rank (see \cite{lu}) of $\mathcal{C}(-m)$ is defined by 
    \begin{equation}\label{rankofclassgrp}
    r_{2^k}(-m)= \dim_{\mathbb{F}_2} 2^{k-1}\mathcal{C}(-m)/2^k \mathcal{C}(-m),
     \end{equation}
     where $\mathbb{F}_2$ is the field containing $2$ elements. Let $D(-m)$ be the discriminant of the quadratic field $\Q (\sqrt{-m})$. Assume that $D(-m)$ has $r$ distinct prime factors in its prime decomposition. From the Gauss genus theory, it is well-known that $r_2(-m)=r-1$. R\`edei studied the $4$-rank of $\mathcal{C}(-m)$ using the Hilbert symbol of rational numbers. For a prime $p$ and integers $a,b$, we write $a=p^\alpha a', b= p^\beta b'$ with $a',b'$ are integers coprime to $p$. For an odd integer $x$, let $w(x):=\frac{x^2-1}{8}$. Then the Hilbert symbol of $a,b$ is defined as
    \begin{equation*}
    \legendre{a,b}{p}= \begin{cases}
        {(-1)}^{\frac{a'-1}{2} \frac{b'-1}{2} + \beta \omega(a') + \alpha \omega(b')} \text{ if } p=2, \\
       (-1)^{\frac{p-1}{2}\alpha \beta} \legendre{a'}{p}^\beta \legendre{b'}{p}^\alpha \quad \, \text{ otherwise. } 
       
    \end{cases}
    \end{equation*}
  \subsection{Relation between 4-rank and 8-rank of class groups} We now concentrate on the square-free natural numbers $n=p_1p_2\cdots p_tq$  with $p_i \equiv 1 \pmod 8$ and $q \equiv 3 \pmod 8$ are prime numbers and $t \geq 1$. Let $\epsilon: \{ \pm 1 \} \to \mathbb{F}_2$ be the group homomorphism defined by $\epsilon(1)=0$ and $\epsilon(-1)= 1$. Then the modified R\`edei martix \cite[\S 2.2]{lu} $R_n$ (for the $4$-rank) over the field $\mathbb{F}_2$ is defined to be \[ R_n= \Bigg(\epsilon \legendre{p_i,-n}{p_j}\Bigg)_{1 \leq i \leq t, 1 \leq j \leq t}.\]

\begin{proposition}\cite[Proposition 2.4]{lu} Suppose that $n$ is defined as in Theorem \ref{thmPq}. Then the $4$-rank of $\mathcal{C}(-n)$ is given by
		$r_4(-n)= t - {\operatorname{rank}}_{\mathbb{F}_2}(R_n) $.
	\end{proposition}
	
\begin{remark}
\label{r41}
We observe that if $n$ is defined as in Theorem \ref{thmPq}, then the modified R\`edei matrix $R_n$ is same as the matrix $A_n$ as in (\ref{matrixA}). Hence $r_4(-n)=1$.
\end{remark}
 

Let $V=\mathcal{C}(-n)[2]$ be the subgroup of $\mathcal{C}(-n)$ consisting of elements of order dividing $2$. The ideal class of the prime ideals $\mathfrak{p}_i, \mathfrak{q}$ of $\Q(\sqrt{-n})$ where $\mathfrak{p}_i^2=(p_i), \mathfrak{q}^2=(q)$ generate $V$. Since the only non-trivial relation among these elements is \begin{equation}\label{nontrivialrelation}
    \mathfrak{p}_1\mathfrak{p}_2\cdots \mathfrak{p}_t\mathfrak{q}=1,
\end{equation}
we see that $V$ forms a vector space over $\mathbb{F}_2$ of dimension $t$. Write $V_0=\mathcal{C}(-n)[2] \cap 2\mathcal{C}(-n)$. It can be checked that $V_0$ is a subspace of $V$ and $\dim_{\mathbb{F}_2}(V_0)=r_4(-n)$ \cite[\S 2.2]{lu}. 
Let $p$ be a prime such that $p\,|\,n$. Then the quadratic characters of $\mathcal{C}(-n)$ are defined by 
\begin{equation}\label{quadratic character}
\chi_{p}(\mathfrak{a})= \legendre{N\mathfrak{a},-n}{p}, 
\end{equation}
where $\mathfrak{a}$ is a fractional ideal of $\Q(\sqrt{-n})$ and $N(\mathfrak{a})$ denotes the norm of the ideal $\mathfrak{a}$. 
\begin{remark}
    The elements of $V_0$ are the elements of $\mathcal{C}(-n)[2]$ that are killed by all the quadratic characters of $\mathcal{C}(-n)$.
\end{remark}

\noindent For more details, we refer to \cite{lu}. Let $V'= \{d>0: d\,|\,n_q \}$ where $n_q=\frac{n}{q}$. We note that $n_q$ has $t$ number of distinct primes in its prime decomposition. For any two elements $a,b \in V'$, we define $a*b$ as in (\ref{defineG}). Then $V' \cong \left(\mathbb{Z}/ 2\mathbb{Z}\right)^{t}$ and hence $V'$ can be seen as a $\mathbb{F}_2$ vector space. 
Furthermore,  
    \begin{equation}\label{isomV'V}
       \begin{split}
         \psi : V' & \to V \\ 
         \psi (p_i) & = \mathfrak{p}_i
       \end{split} 
    \end{equation}
    is a vector space isomorphism. 
    Let $V_0'$ be the subspace of $V'$ identified with the preimage of $V_0 \subset V$ under the map $\psi$. Let $\mathcal{B}: V_0 \times V_0 \to \mathbb{F}_2$ be the bilinear form \cite[\S 2.3]{lu} defined by 
    \begin{equation}
    \label{bilinearf}
        \mathcal{B}(\mathfrak{a},D) = \epsilon\bigg(\legendre{N\mathfrak{b},-n}{D}\bigg) = \epsilon\bigg( \prod\limits_{p|D}\legendre{N\mathfrak{b},-n}{p}\bigg) 
    \end{equation}
    where $\mathfrak{b}$ is the fractional ideal such that $\mathfrak{b}^2$ and $\mathfrak{a}$ represent the same class in $\mathcal{C}(-n)$ and $D \in V_0'$ since $V_0' \cong V_0$. The following proposition relates the $4$-rank and $8$-rank of $\mathcal{C}(-n)$.
    \begin{proposition}\cite[Proposition 2.7.]{lu} \label{rebe4rk8rk}
        $r_8(-n)=r_4(-n)- {\operatorname{rank }}_{\mathbb{F}_2}\mathcal{B}$.
    \end{proposition}


	
	

  Now we provide a condition for computing the 8-rank $r_{8}(-n)$. Before stating it, consider the following: Let  $l$ be a prime of the form $4l'+1$, and let $k$ be any integer that is a quadratic residue modulo $l$. We denote the quartic symbol by $\legendre[4]{k}{l}$, which takes the value $\pm 1$, with a positive sign if and only if $k$  is a quartic residue modulo $l$.
\begin{lemma}\label{r4r8n}
Let $n$ be a number satisfying the properties of Theorem \ref{thmPq}. Then $r_8(-n)=1$ if and only if $\legendre[4]{q}{n_q}=1$. 
	\end{lemma}
	
	\begin{proof}
        Since $r_4(-n) = 1$ (by Remark \ref{r41}), there are exactly two elements in $V_0$ as  $\dim_{\mathbb{F}_2}(V_0)=r_4(-n)=1$. Consider the following two ideals $$\mathfrak{P}_0=\mathfrak{q}\prod\limits_{i=1}^{t}\mathfrak{p}_i,\quad \mathfrak{P}_1=\prod\limits_{i=1}^{t}\mathfrak{p}_i,$$ where each $\mathfrak{q}, \mathfrak{p}_i$  are  prime ideals lying above $q$ and $p_i$ in $\Q(\sqrt{-n})$, respectively, for $1 \leq i \leq t$. From (\ref{quadratic character}), we have  $\chi_{p_i}(\mathfrak{P}_j)=\chi_q(\mathfrak{P}_j)=1$ for  all $1 \leq i \leq t$,  and  $j \in \{0,1\}$. Hence both $\mathfrak{P}_0, \mathfrak{P}_1\in V_0$. 
        From the relation in (\ref{nontrivialrelation}), we have $\mathfrak{P}_0= \mathfrak{q}\prod\limits_{i=1}^{t}\mathfrak{p}_i=1$, in other words, $\mathfrak{P}_0$ is the trivial element of $V_0$. Hence it follows that $\mathfrak{P}_1$ is the unique non-trivial element of $V_0$. Identifying $V$ with $V'$ as discussed above (\ref{isomV'V}), it is easy to see that $\mathfrak{P}_1$ corresponds to $n_q$ in $V'$. \smallskip

        First assume that $r_8(-n)=1$. By Proposition \ref{rebe4rk8rk}, $r_8(-n)=r_4(-n)- \text{rank}_{\mathbb{F}_2}\mathcal{B}$ where $\mathcal{B}$ is the bilinear form defined in (\ref{bilinearf}). Consider the quadratic form $Q_{\mathcal{B}}: V_0 \to \mathbb{F}_2$ such that $ \mathcal{B}(D_1,D_2)= Q_{\mathcal{B}}(D_1+D_2) - Q_{\mathcal{B}}(D_1) - Q_{\mathcal{B}}(D_2)$ (since $V_0 \cong V_0'$). 
        Since $r_8(-n)=r_4(-n)=1$, the bilinear form $\mathcal{B}$ is zero and hence $Q_\mathcal{B}=0$ on $V_0$. 
        Now using Theorem 3.4 of \cite{lu}, 
        we get $Q_\mathfrak{B}(n_q)= \epsilon \big(\legendre[4]{-n/n_q}{n_q}\big)=0$ i.e. $\legendre[4]{-q}{n_q}=1 \implies \legendre[4]{q}{n_q}=1$ since all the prime factors of  $n_q$ are congruent 1 modulo 8.\smallskip

		Next assume that $\legendre[4]{q}{n_q}=1$. Since $V_0$ has only two elements, using \cite[Th. 3.4]{lu} we conclude that the quadratic form $Q_{\mathcal{B}}=0$ on $V_0$ which implies that the bilinear form $\mathcal{B}=0$. Hence by Proposition \ref{rebe4rk8rk}, we have that $r_8(-n)=r_4(-n)=1$. 
  \end{proof}
        
        Now we write down a condition for the non-trivial $4$-rank and $8$-rank of the class group of $\Q(\sqrt{-n_q})$ where $n_q=\frac{n}{q}$. Since $n_q=P=p_1p_2\cdots p_t$ with $p_i \equiv 1 \pmod 8$, $n_q$ can be written as $n_q = 2e^2-f^2$ (with $e,f$ both are odd) as discussed in (\ref{expP}). 
        \begin{proposition}\cite[Theorem 1]{wy}\label{r4r8nq}
     Let $n=n_{q}q$ be defined as in Theorem \ref{thmPq}.
     Assume that $\operatorname{rank} A_n=t-1$, where $A_{n}$ is given as in (\ref{matrixA}). Then we have $r_4(-n_q)=1$. Moreover, $r_8(-n_q)=1$ if and only if $\legendre{-1}{e}=1$. 
        \end{proposition}

\section{Proof of the main theorem}
\label{proofofthemaintheorem}
\noindent Now we are in the state of proving the main theorem which uses the results obtained in the previous sections.

	\begin{proof}[Proof of Theorem \ref{thmPq}]
            We observe that the discriminant $D(-n_q)$ of quadratic field of $\mathbb{Q}(\sqrt{-n_{q}})$ is $-4n_{q}$ since, $-n_{q} \equiv 3 \pmod{4}$ and hence $r_2(-n_q)=t$ (by Gauss genus theory). Since $r_4(-n_q)=1$, we have $h(-n_q) \equiv 0 \pmod{2^{t+1}}$ by Proposition \ref{r4r8nq}. Also from Proposition \ref{r4r8nq}, we have $h(-n_q) \equiv 0 \pmod{2^{t+2}}$ if and only if $\legendre{-1}{e}=1$. Using Lemma \ref{r4r8n}, we have $h(-n) \equiv 0 \pmod{2^{t+2}}$ if and only if $\legendre[4]{q}{n_q}=1$. Hence it is enough to show that if $n$ is a congruent number,  
            then $\legendre{-1}{e}= \legendre[4]{q}{n_q}$.  \smallskip

Suppose that $n$ is a congruent number. Hence $E_n(\Q)$ has a point of infinite order, or the Mordell-Weil rank $r_{n}$ of $E_{n}$ over $\Q$ is positive. By Lemmas \ref{expressK}, \ref{dequa_rank_Selmer} and  Remark \ref{rem_Sel_to_Mordell}  we can say that the system of equations (\ref{congeq}) has non-trivial integral solutions for at least one pair $(a, b) \in \{(1,n_q),(n_q,1),(n_q,n_q)\}$. By Proposition \ref{r4r8nq} and Lemma \ref{r4r8n}, it is enough to show that for each of the above cases, the following equivalent condition holds: $\legendre{-1}{e} = 1 \iff \legendre[4]{q}{n_{q}} = 1$.\smallskip

\noindent \textbf{(a)} Let $(a,b)=(1,n_q)$\label{(a)}. Substituting  $(a,b)=(1,n_q)$ and $m=n$ in (\ref{congeq}), we get 
    \begin{equation}
   	\begin{split}
   	n_qx^2+ny^2 & = z^2, \\
    n_qx^2-ny^2 & = n_qw^2.
    \end{split}
    \end{equation}
Note that $n_q\,|\,z$. Dividing the above equations by $n_q$ and denoting $\frac{z}{n_q}$ by $z$ again, we get 
\begin{equation}\label{congeq1nq1}
    x^2+qy^2=n_qz^2,
\end{equation}
\begin{equation}\label{congeq1nq2}
    x^2-qy^2=w^2.
\end{equation}
Adding the above equations, we obtain 
    \begin{equation}\label{congeq1nqadd}
    2x^2 = n_qz^2 + w^2.
    \end{equation} 
Observe that $x,z,w$ are odd and $2\,|\,y$. Considering (\ref{congeq1nq1}) modulo $n_q$, we note that $x^2\equiv -qy^2 \pmod{n_q}$. Hence $\legendre[4]{x^2}{n_q}=\legendre[4]{-q}{n_q}\legendre[4]{y^2}{n_q} \implies \legendre{x}{n_q}\legendre{y}{n_q}=\legendre[4]{q}{n_q}$ (since $\legendre[4]{-1}{n_q}=1$). 
Let $y = 2^{a_0}y_0$ with $y_0$ being an odd integer. From (\ref{congeq1nq1}), modulo  $y_0$, we have $\legendre{n_q}{y_0} = 1$, which implies $\legendre{y_0}{n_q} = 1$. Observe that $\legendre{2}{n_q} = 1$ because all prime divisors of $n_q$ are congruent to $1 \bmod 8$. Consequently, we have $\legendre{x}{n_q} = \legendre[4]{q}{n_q}$. Moreover,  (\ref{congeq1nqadd}) yields the congruence $n_q z^2 \equiv -w^2 \pmod{x}$, which implies that $\legendre{n_q}{x} = \legendre{-1}{x}$. Hence \begin{equation}\label{rebeqnqx}
\legendre[4]{q}{n_q}=\legendre{x}{n_{q}}=\legendre{n_{q}}{x}=\legendre{-1}{x}.
\end{equation}
We recall that $n_q$ can be expressed as $u^2+2v^2$ as discussed in (\ref{expP}). Substituting this into (\ref{congeq1nqadd}), we have $2(x+vz)(x-vz)=u^2z^2+w^2$. 
Assume that the $\gcd(u,w)=d >1$. Let $\ell$ be an odd prime such that $\ell~|~ u^2z^2+w^2$. If $\ell \, | \, d$, then $\ell^2 \, | \, (x+vz)(x-vz)$. Also $\gcd(x,z)=1$ implies that $\ell^2$ divides exactly one of $x+vz$ or $x-vz$. Note that  $\ell^2 \equiv 1 \pmod 4$. Let $\ell \nmid d$. Since $u^2z^2 \equiv -w^2 \pmod \ell$, it implies that $\ell\equiv 1 \pmod 4$. Hence $x\pm vz \equiv 1 \pmod 4$. We recall that $z$ is odd, therefore, $x \equiv 1 \pmod{4}$ if and only if $v \equiv 0 \pmod 4$. From Lemma \ref{rebeeandv}, we have $\legendre{-1}{e}=1$ if and only if $v \equiv 0 \pmod 4$. 
Hence from (\ref{rebeqnqx}), we have the following: 
\begin{enumerate}

    \item $\legendre[4]{q}{n_q}=1 \iff x\equiv 1 \pmod 4 \iff v \equiv 0 \pmod 4 \iff \legendre{-1}{e}=1$,

    \item $\legendre[4]{q}{n_q}=-1 \iff x\equiv 3 \pmod 4 \iff v \equiv 2 \pmod 4 \iff \legendre{-1}{e}=-1$.
\end{enumerate}\smallskip
Hence we conclude that in this case  $\legendre{-1}{e} = 1 \iff \legendre[4]{q}{n_{q}} = 1$. \\

\noindent \textbf{(b)} Let $(a,b)=(n_q,1)$. Substituting the value of $(a,b)$ in (\ref{congeq}), we observe that $n_q \,|\,w$. We divide both the equations by $n_q$ (and denote $\frac{w}{n_q}$ by $w$) which gives 
    \begin{equation}\label{congeqnq11}
        x^2+qy^2 = z^2,
    \end{equation}
    \begin{equation}\label{congeqnq12}
        x^2-qy^2 = n_qw^2.
    \end{equation}
Adding and subtracting the above equations, we get the following
\begin{equation}\label{congeqnq1add}
        2x^2 = z^2+ n_qw^2,
    \end{equation}
    \begin{equation}\label{congeqnq1sub}
        2qy^2 = z^2-n_qw^2.
    \end{equation}
Note that (\ref{congeqnq1add}) implies $x,z,w$ are odd and from (\ref{congeqnq1sub}), we get $y$ to be even. From (\ref{congeqnq12}) and by the similar argument used in the previous case, we obtain $\legendre{x}{n_q}=\legendre[4]{q}{n_q}$. Further from (\ref{congeqnq1add}) we have that $\legendre{n_q}{x}=\legendre{-1}{x}$ which implies that $\legendre[4]{q}{n_q}=\legendre{-1}{x}$. Using the fact that $n_q=u^2+2v^2$, where $u$ and $v$ are positive integers (as in (\ref{expP})), we get from (\ref{congeqnq1add}) that $2(x+vw)(x-vw)=z^2+u^2w^2$. It can be shown (using the similar argument used as in the case $(a,b)=(n_q,1)$) that 
$x\pm vw \equiv 1 \pmod 4$. Since $w$ is odd, we use Lemma \ref{rebeeandv} to conclude the following:
\begin{enumerate}
    \item $\legendre[4]{q}{n_q}=1 \iff x\equiv 1 \pmod 4$, then $v \equiv 0 \pmod 4$ i.e., $\legendre{-1}{e}=1$,

    \item $\legendre[4]{q}{n_q}=-1 \iff x\equiv 3 \pmod 4$, then $v \equiv 2 \pmod 4$ i.e., $\legendre{-1}{e}=-1$.
\end{enumerate}\smallskip

\noindent \textbf{(c)} Let $(a,b)=(n_q,n_q)$. As in the previous cases, we substitute the value of $(a,b)$ in (\ref{congeq}) and divide both the equations by $n_q$, 
which gives 
   \begin{equation}\label{congeqnqnq1}
   	n_qx^2+qy^2 = z^2, 
   \end{equation}
   \begin{equation}\label{congeqnqnqq}
        n_qx^2-qy^2 = w^2.
   \end{equation}
Adding the above two equations, we get 
\begin{equation}\label{congeqnqnqadd}
   	2n_qx^2 = z^2 + w^2.
   \end{equation}
Here, $x,z,w$ are odd and $y$ is even. From (\ref{congeqnqnq1}), we also have $\legendre[4]{q}{n_q}\legendre{y}{n_q}=\legendre{z}{n_q}$. Modulo $z$,  (\ref{congeqnqnqadd}) implies that $\legendre{n_q}{z}=\legendre{2}{z}$. Since $y$ is even, we write $y=2^{a_0}y_0$, with $2 \, \nmid \, y_0$. Hence from (\ref{congeqnqnq1}), we have $\legendre{n_q}{y_0}=1$. Using all these together with the fact that $\legendre{n_q}{y_0}=\legendre{y_0}{n_q}$, we obtain
\begin{equation} \label{rebeqnqz}
\begin{split}
 & \legendre[4]{q}{n_q}\legendre{y}{n_q}  =\legendre{z}{n_q}, \\ \implies  & \legendre[4]{q}{n_q} \legendre{y_0}{n_q}\legendre{2}{n_q}^{a_0}  =\legendre{2}{z}, \\ 
 \implies &  \legendre[4]{q}{n_q}=\legendre{2}{z}.
\end{split}
\end{equation}
Substituting $n_q=2e^2-f^2$ for some positive integers $e,f$ (as in (\ref{expP})), from (\ref{congeqnqnqadd}) we get $(2ex+z)(2ex-z)=w^2+2f^2x^2$. Since $f$, $w$, and $x$ are all odd, it follows that $w^2 + 2f^2x^2 \equiv 3 \pmod{8}$. Moreover, from equation (\ref{congeqnqnqadd}), we deduce that $\legendre{-1}{x} = 1$, which implies $x \equiv 1 \pmod{4}$. Let $\ell$ be an odd prime factor of $w^2+2f^2x^2$ such that $\ell  \nmid \gcd(w,f)$.  Then $\legendre{-2}{\ell}=1$ which implies that $\ell \equiv 1$ or, $3 \pmod 8$. If $\ell  \, | \, \gcd(w,f)$, then $\ell^2$ divides exactly one of  $2ex+z$, and $2ex-z$. \smallskip

Let $2ex+z = r_1^2l_1l_2\cdots l_s$ and $2ex-z=r_2^2l_{s+1}l_{s+2}\cdots l_k$ be the factorization of $2ex+z$ and $2ex-z$ respectively where $l_{i}$'s are primes, and $\gcd(l_i,r_j)=1$, for all $i \in \{1,2, \ldots, k\},$ and $j \in \{1,2\}$. Note that, for $j\in \{1,2\}$, we have $r_j^2 \equiv 1 \pmod 8$, and hence one of the following two conditions holds: 

\begin{enumerate}
    \item $2ex+z \equiv 1 \pmod 8$ and $2ex-z \equiv 3 \pmod 8$,

    \item $2ex+z \equiv 3 \pmod 8$ and $2ex-z \equiv 1 \pmod 8$.
\end{enumerate} \smallskip

Let $\legendre{2}{z}=1$ i.e., $z \equiv 1$ or, $7 \pmod 8$. If $z \equiv 1 \pmod 8$, then (2) holds and hence $e \equiv 1 \pmod 4 \implies \legendre{-1}{e}=1$. If $z \equiv 7 \pmod 8$, then (1) holds which implies that $\legendre{-1}{e}=1$. If $\legendre{2}{z}=-1$, then $z \equiv 3$ or, $5 \pmod 8$. Using the similar arguments as in the previous case, we get $\legendre{-1}{e}=-1$. Now using (\ref{rebeqnqz}), we have $\legendre[4]{q}{n_q} = \legendre{-1}{e}$.








		
	\end{proof} 
	
 Observe that if $t=1$, then $\operatorname{rank} A_n=0$ (see (\ref{matrixA})). Hence we have the following: 
	
\begin{corollary}\label{t=1}
    Let $n=pq$ be a congruent number with $(p,q) \equiv (1,3) \pmod 8$ and $\legendre{q}{p}=1$, then $h(-pq) \equiv h(-p) \pmod{8}$.
\end{corollary}	

\begin{remark}
 This shows that, for $t=1$, Theorem \ref{thmPq} coincides with Theorem 1.6 of \cite{qin} and generalises this result for  product of arbitrarily many primes $p_i \equiv 1 \pmod 8$.
\end{remark}



    
    

\section{Examples}\label{example}
In order to demonstrate our results, we find some congruent numbers and applying the results in Theorem \ref{thmPq}; we verify our computation by SageMath \cite{sage}.
\begin{enumerate}
    \item[1.] Table \ref{table:Pq-cong} and Table \ref{table:Pq-non-cong} correspond to Theorem \ref{thmPq}. We first search for all integers up to $500{,}000$ that satisfy the hypothesis of Theorem~\ref{thmPq} in the case $t=2$. We write the congruent numbers (implying that the congruence condition $h(-n) \equiv h(-n_{q}) \pmod{16}$ is satisfied) in Table \ref{table:Pq-cong}.

\item[2.] By checking whether the congruence condition $h(-n) \not \equiv h(-n_{q}) \pmod{16}$ is satisfied, we sort out the non-congruent numbers, which are listed in Table \ref{table:Pq-non-cong}. 
    
    \item[3.] In fact,  we observe that there exist some non-congruent numbers such that the congruence condition $h(-n) \equiv h(-n_{q}) \pmod{16}$ is satisfied. For example, $n=68547$, $110627$, $126363$, $167907$ and so on; this is verified by Sage \cite{sage}. This shows that  Theorem \ref{thmPq} provides a necessary condition for a number to be a congruent number which is not sufficient.


\end{enumerate}
\begin{table}[ht]
\begin{center}
\caption{Theorem \ref{thmPq}: Congruent numbers}\label{table:Pq-cong}  
\begin{tabular}{|l|l|l|l|l|l|}
\hline
$n$  & $q$ & $p_1\cdot p_2$   &  $\left( \legendre{q}{p_1}, \legendre{q}{p_2}, \legendre{p_{1}}{p_{2}} \right)$ &$h(-n)$ & $h(-p_1p_2)$\\
\hline
$52779$   & $3$ & $73 \cdot 241$  &  $(1,~\,1,~\,-1)$ & $80$  & $48$   \\ \hline
$134123$  & $11$ & $89 \cdot 137$    & $(1,~\,1,~\,-1)$  & $88$ & $56$   \\ \hline
$220971$  & $3$ & $73 \cdot 1009$     & $(1,~\,1,~\,-1)$  & $96$  & $128$   \\ \hline
$263019$  &  $3$ & $73 \cdot 1201$  & $(1,~\,1,~\,-1)$ & $104$ & 
$152$   \\ \hline
$264603$  &  $3$ & $193 \cdot 457$  & $(1,~\,1,~\,-1)$ & $104$ & $184$   \\ \hline
$274219$  &  $11$ & $97 \cdot 257$  & $(1,~\,1,~\,-1)$ & $88$ & $136$   \\ \hline
$289299$  &  $3$ & $73 \cdot 1321$   & $(1,~\,1,~\,-1)$  & $168$ & $152$  \\ \hline
$326091$  &  $3$ & $73 \cdot 1489$  &  $(1,~\,1,~\,-1)$ & $224$ &  $160$  \\ \hline
$417171$ &  $3$ & $241 \cdot 577$ & $(1,~\,1,~\,-1)$   & $264$ & $168$ \\ \hline
$462011$ &  $11$ & $97 \cdot 433$ & $(1,~\,1,~\,-1)$   & $256$ & $80$ \\ \hline
$468003$ &  $3$ & $73 \cdot 2137$ &  $(1,~\,1,~\,-1)$  & $152$ & $312$ \\ \hline
\end{tabular}
\end{center}
\end{table}


\begin{table}[ht]
\begin{center}
\caption{Theorem \ref{thmPq}: Non-congruent numbers}\label{table:Pq-non-cong}
\begin{tabular}{|l|l|l|l|l|l|}
\hline
$n$  & $q$ & $p_1\cdot p_2$   &  $\left( \legendre{q}{p_1}, \legendre{q}{p_2}, \legendre{p_{1}}{p_{2}} \right)$ &$h(-n)$ & $h(-p_1p_2)$\\
\hline
$42267$   & $3$ & $73 \cdot 193$  &  $(1,~\,1,~\,-1)$ & $24$  & $96$   \\ \hline
$73803$  & $3$ & $73 \cdot 337$   &  $(1,~\,1,~\,-1)$ & $32$ &   $120$ \\ \hline
$89571$  & $3$ & $73 \cdot 409$    & $(1,~\,1,~\,-1)$  & $80$ & $56$   \\ \hline
$94827$  & $3$ & $73 \cdot 433$     & $(1,~\,1,~\,-1)$  & $72$  & $96$   \\ \hline
$98067$  &  $3$ & $97 \cdot 337$  & $(1,~\,1,~\,-1)$ & $48$ & 
$104$   \\ \hline
$119019$  &  $3$ & $97 \cdot 409$  & $(1,~\,1,~\,-1)$ & $144$ & $88$   \\ \hline
$126003$  &  $3$ & $97 \cdot 433$  & $(1,~\,1,~\,-1)$ & $56$ & $80$   \\ \hline
$131619$  &  $3$ & $73 \cdot 601 $ &  $(1,~\,1,~\,-1)$ & $88$ & $64$  \\ \hline
$132987$  &  $3$ & $97 \cdot 457$   & $(1,~\,1,~\,-1)$  & $64$ & $168$  \\ \hline
$146179$  &  $11$ & $97 \cdot 137$  &  $(1,~\,1,~\,-1)$ & $80$ &  $168$  \\ \hline
\end{tabular}
\end{center}
\end{table}

\newpage	

\noindent{\it Acknowledgement}: The first author acknowledges support from the DST–INSPIRE Faculty Fellowship at IIT Kanpur. The second author acknowledges the support provided by the Institute Postdoctoral Fellowship at HRI Prayagraj, where the majority of the work has been conducted, and the support provided by the Institute Postdoctoral Fellowship at IIT Madras.

\nocite{*}
\bibliographystyle{siam}
\bibliography{Congruent_Class}

\end{document}